\documentclass[reqno,11pt]{article}
\usepackage{a4wide,xcolor,eucal,enumerate,mathrsfs}
\usepackage[normalem]{ulem}
\usepackage{amsmath,amssymb,epsfig,amsthm}%bbm,
\usepackage[latin1]{inputenc}

\newcommand{\N}{\mathbb{N}}

\newcommand{\R}{\mathbb{R}}

\newcommand{\mm}{{\mbox{\boldmath$m$}}}

\newcommand{\sfd}{{\sf d}}

\newcommand{\Kliminf}{K\kern-3pt-\kern-2pt\mathop{\rm lim\,inf}\limits}  % Kuratowski liminf di insiemi
\newcommand{\supp}{\mathop{\rm supp}\nolimits}   % supporto 
\renewcommand{\d}{{\mathrm d}}

\newcommand{\restr}[1]{\lower3pt\hbox{$|_{#1}$}}

\newcommand{\down}{\downarrow}              %frecce in su e in giu nei limiti
\newcommand{\up}{\uparrow}
  
\newcommand{\nchi}{{\raise.3ex\hbox{$\chi$}}}

\newcommand{\Pc}[2]{\overline{#1}\kern-2pt^{\vphantom 0}_{#2}}
\numberwithin{equation}{section}

\newtheorem{theorem}{Theorem}[section]

\newtheorem{lemma}[theorem]{Lemma}
\newtheorem{proposition}[theorem]{Proposition}

\newtheorem{definition}[theorem]{Definition}

\newtheorem{example}[theorem]{Example}
\newtheorem{remark}[theorem]{Remark}

\renewcommand{\mm}{\mathfrak m}

\newcommand{\LIP}{{\rm LIP}}
\newcommand{\Ric}{{\rm Ric}}
\newcommand{\length}{{\rm length}}
\newcommand{\mms}{{\rm m.m.s.}}
\newcommand{\fr}{\hfill$\blacksquare$}                      %quadratino nero alla fine del remark, se non vi piace, la cosa migliore e' `svuotare' la macro, cosi' non bisogna intervenire sul testo

\title{A new notion of angle between three points in a metric space}
\begin{document}

\author{%Luigi Ambrosio\
  % \thanks{Scuola Normale Superiore, Pisa, \textsf{l.ambrosio@sns.it}}
   %\and
   %Nicola Gigli\
   %\thanks{University of Nice, \textsf{nicola.gigli@unice.fr}}
    Andrea Mondino
    \thanks{Scuola Normale Superiore, Pisa, \textsf{andrea.mondino@sns.it}}
   %\and
   %Tapio Rajala
   %\thanks{University of Jyv\"askyl\"a, \textsf{tapio.m.rajala@jyu.fi}}
   %\and
   %Giuseppe Savar\'e\
   %\thanks{University of Pavia, \textsf{giuseppe.savare@unipv.it}}
   }

\maketitle

\begin{abstract}
We give a new notion of angle in general metric spaces; more precisely, given a triple a points $p,x,q$ in a metric space $(X,\sfd)$, we introduce the notion of angle cone $\bold{\angle_{pxq}}$ as being an interval $\bold{\angle_{pxq}}:=[\angle^-_{pxq},\angle^+_{pxq}]$, where the quantities $\angle^\pm_{pxq}$ are defined in terms of the distance functions from $p$ and $q$ via a duality construction of differentials and gradients holding for locally Lipschitz functions on a general metric space. Our definition in the Euclidean plane gives the standard angle between three points and in a Riemannian manifold coincides with the usual angle between the geodesics, if $x$ is not in the cut locus of $p$ or $q$. We show that in general the angle cone is not single valued (even in case the metric space is a smooth Riemannian manifold, if $x$ is in the cut locus of $p$ or $q$), but if we endow the metric space with a positive Borel measure $\mm$ obtaining the metric measure space $(X,\sfd,\mm)$ then under quite general assumptions (which include many fundamental examples as Riemannian manifolds, finite dimensional Alexandrov spaces with curvature bounded from below, Gromov-Hausdorff limits of Riemannian manifolds with Ricci curvature bounded from below, and normed spaces with strictly convex norm), fixed $p,q \in X$, the angle cone at $x$ is single valued for $\mm$-a.e. $x \in X$. We prove some basic properties of the angle cone (such as the invariance under homotheties of the space) and we analyze in detail the case $(X,\sfd,\mm)$ is a measured-Gromov-Hausdorff limit of a sequence of Riemannian manifolds with Ricci curvature bounded from below, showing the consistency of our definition  with a recent construction of Honda \cite{Honda12}.  
\end{abstract}

\tableofcontents

\section{Introduction}
The notion of \underline{angle}, together with the one of \underline{distance} (which is directly linked with one of length), is at the basis of geometry from both the conceptual and the historical point of view, let us just give two fundamental examples.
\begin{itemize}
\item \underline{Euclidean geometry}: the last two of the  five classical postulates of Euclid for plane geometry (see the  beginning of the first book of the \textquotedblleft Elements\textquotedblright, pp.195-202 in \cite{Euclide}), as well as the Hilbert's axiomatization of Euclidean geometry (see \cite{Hilbert}), are based on the notion of angle.
\item \underline{Riemannian geometry}: in the very first definition of a Riemannian manifold one assigns a scalar product on the tangent spaces, i.e. the length and the angle of tangent vectors. 
\end{itemize}
The natural framework in which it is meaningful to speak of  distance between couples of points  is given by the class of  metric spaces; the goal of this paper is to give a definition of angle between three points in a metric space  in terms just of the distance.% (in the second part of the paper we will study the case of a metric measure space, i.e.  a metric space endowed with a  positive Borel measure, and the angle will depend just on the distance and on this measure).
\medskip

Some early attempts to give a definition of angle between three points in a metric space go back to the work of Wilson \cite{Wilson} in the 30's (see also  \cite{Valentine}); the drawback of this approach  is that  the definition was given under  very particular convexity and embeddability assumptions  (satisfied in linear spaces but, for instance, what Wilson called \textquotedblleft the four point property\textquotedblright is at the base of his definition and does not hold for the round sphere $S^2\subset \R^3$). 

In what are now known as Alexandrov spaces (these are metric spaces satisfying curvature assumptions, which in case of smooth Riemannian manifolds correspond to upper or lower bounds on the sectional curvature; there is a huge literature about this argument, the interested reader may see \cite{BBY}), Alexandrov  was able to define the angle between two geodesics matching at a point as follows: let $\gamma_1,\gamma_2$ be unit speed geodesics\footnote{in our terminology the geodesics are length minimizing by definition, see \eqref{defgeo}}  such that $\gamma_1(0)=\gamma_2(0)$, then he set 
\begin{equation}\label{eq:angleAlex}
\angle(\gamma_1, \gamma_2):= \arccos \left(\lim_{s,t\to 0} \frac{s^2+t^2-\sfd^2(\gamma_1(s), \gamma_2(t))}{2st} \right); 
\end{equation}
the crucial observation being that  the curvature bounds ensure the existence of the limit above (fact which may not be true in a general metric space), hence the definition makes sense. %Later in the introduction we will discuss a bit more  the case of Alexandrov spaces (see Example \ref{Example:AlexSpace}), for the moment let us just stress that our new definition holds in general metric spaces and no curvature assumption is needed.

To conclude the brief historical overview, recently, Honda \cite{Honda12} proved that the definition of angle \eqref{eq:angleAlex} can be suitably adapted to the case  the metric space is a Gromov-Hausdorff limit of a sequence of Riemannian manifolds with Ricci curvature bounded from below (the theory of such spaces was started by Cheeger and Colding \cite{ChCo3}). %In the end of the paper we will prove that, in the particular case the metric space is such a limit space, our notion of angle coincides with the one of Honda almost everywhere, see Theorem

Let us stress that our new definition holds in general metric spaces and no curvature assumption is needed.
\medskip

Our idea for defining the angle between three points in a metric space is to use the distance functions and to  adapt the duality between differentials and gradients introduced by Gigli in \cite{G12}. Let us start with an  example in the smooth case.  

\begin{example}[The angle between three points in a Riemannian manifold]\label{ex:angleRiem}
\rm{ Let $(M^n,g)$ be a smooth $n$-dimensional Riemannian manifold and fix $p,q \in M$. The goal is to define, for  $x \in M$, the angle $\angle_{pxq}$ between $p$ and $q$ at $x$ in terms of the distance functions $r_p(\cdot):=\sfd(p,\cdot)$ and $r_q(\cdot):=\sfd(q,\cdot)$, where $\sfd(\cdot,\cdot)$ is the distance induced by the Riemannian metric. 

Called $C_p$ the cut locus of the point $p$, notice that the distance function $r_p$ is smooth on $M\setminus(C_p\cup \{p\})$ and moreover there is a unique unit speed geodesic$^1$   from $x$ to $p$, for $x \in M\setminus(C_p\cup \{p\})$; hence, for $x \in M\setminus(C_p \cup C_q \cup\{p,q\})$ we can set 
\begin{equation}\label{eq:angleRiem}
\angle_{pxq}:=\arccos\left[g_x(\nabla r_p(x), \nabla r_q(x))\right]=\arccos\left[g_x(\dot{\gamma}_{xp}(0), \dot{\gamma}_{xq}(0))\right]=\angle(\gamma_{xp},\gamma_{xq}),
\end{equation}
where $\nabla r_p(x)$ is the gradient of the distance function $r_p$ computed at the point $x$, and $\gamma_{xp}$ (resp. $\gamma_{xq}$) is the unique unit speed geodesic from $x$ to $p$ (resp. from $x$ to $q$).
Observe that in the case of the euclidean plane $\R^2$, the definition \eqref{eq:angleRiem} gives the usual angle between  $p$ and $q$ at $x$. 

Recall that in general the cut locus is non empty (for instance if $M$ is compact the cut locus is never empty), but it has null measure with respect to the $n$-dimensional Riemannian  volume measure (more precisely it is $(n-1)$-rectifiable, see for instance \cite{MantMenn}), therefore \eqref{eq:angleRiem} does not hold in general for every $x \in M$: indeed if $x\in C_p\cup C_q$ then the distance functions are no more differentiable, there can be more geodesic connecting $x$ to $p$ or $q$, hence  definition \eqref{eq:angleRiem} is no more well posed. Nevertheless we can at least define a notion of \textquotedblleft angle cone\textquotedblright between $p$ and $q$ at $x$ as being the interval $\bold{\angle_{pxq}}:=[\angle^{-}_{pxq}, \angle^+_{pxq}]$  where 
\begin{eqnarray}
\angle^+_{pxq}&=& \sup \{\angle(\gamma_{xp},\gamma_{xq}) : \gamma_{xp}, \gamma_{xq} \text{ are unit speed geod. from $x$ to $p$ and $q$ resp.} \}  \label{eq:angle-mot} \\
\angle^-_{pxq}&=&\inf \{\angle(\gamma_{xp},\gamma_{xq}): \gamma_{xp}, \gamma_{xq} \text{ are unit speed geod. from $x$ to $p$ and $q$ resp.} \} \label{eq:angle+mot}.
\end{eqnarray}
Summarizing, fixed $p,q \in M$, for every $x \in M$ we can define an angle cone $\bold{\angle_{pxq}}$ between $p$ and $q$ at $x$, moreover for a.e. $x \in M$  the angle cone is single valued and we can speak of angle between $p$ and $q$ at $x$.
}\fr 
\end{example}

Our approach for defining the angle between three points in a general metric space is to suitably  understand the term $g_x(\nabla r_p(x), \nabla r_q(x))$ in \eqref{eq:angleRiem}. Notice that in a general  metric space we do not have  notions of differential and gradient of a Lipschitz function (which require a tangent space structure for having a meaning), we have just the local Lipschitz constant (called also slope, the standard definition being  \eqref{eq:slope}) which can be read as \textquotedblleft the modulus of the differential\textquotedblright. The point is that, given two locally Lipschitz functions $f$ and $g$ on the metric space $(X,\sfd)$, the local Lipschitz constant is enough in order to define \textquotedblleft the differential of $g$ acting on the gradient of $f$\textquotedblright; since such an object can be multivalued even in  normed spaces if the norm is not strictly convex (see the discussion in the introduction of \cite{G12}), what is actually meaningful is to define \textquotedblleft the differential of $g$ acting on the gradient of $f$\textquotedblright as an interval $[D^-g(\nabla f), D^+g(\nabla f)]$. This strategy is strongly inspired from a recent work of Gigli \cite{G12} (done in the different framework of weak upper gradients of Sobolev functions) and it is described in Subsection \ref{Subsec:Prel}, in particular see Definition \ref{def:Dfnablag}.

With this tool in our hands and inspired by the smooth case \eqref{eq:angleRiem}-\eqref{eq:angle-mot}-\eqref{eq:angle+mot}, in Subsection \ref{Subsec:DefAngleCone} we  define the angle cone of a triple of points $p,x,q\in X$ in the metric space $(X,\sfd)$ as  the interval $\bold{\angle_{pxq}}:=[\angle^-_{pxq}, \angle^+_{pxq}]$ where 
$$
\angle^-_{pxq}:=\arccos[D^+r_p(\nabla r_q)(x)], \quad \angle^+_{pxq}:=\arccos[D^-r_p(\nabla r_q)(x)],
$$
being $r_p:=\sfd(p,\cdot)$ and $r_q:=\sfd(q,\cdot)$  the distance functions from $p$ and $q$; see Definition \ref{def:anglecone(X,d)} and the discussion before it. The relation with the smooth case is discussed in Example \ref{ex:Mg}. In general we cannot expect the angle cone to be single valued since, as explained in Example \ref{ex:angleRiem}, already in the smooth framework this  may not be the case. There are at least two reasons for this non uniqueness behavior:
\begin{itemize}
\item $x$ is in the cut locus of $p$ or  $q$, see Example \ref{Ex:S2}.
\item Potential lack of strict convexity of the space $(X,\sfd)$, see Example \ref{Ex:RdueLinfty}.  
\end{itemize}
In Subsection \ref{Subsec:Properties} we prove two fundamental properties of the angle cone:
\begin{itemize}
\item If $p=q$ then the angle cone is single valued: for every $x\in X$ it holds $\bold{\angle_{pxp}}=\{\angle_{pxp}\}$. If moreover $(X,\sfd)$ is a length space then for every $p,x \in X$ we have $\bold{\angle_{pxp}}=\{0\}$. See Proposition \ref{prop:pxp=0} and Remark \ref{rem:pxp=0}.
\item The angle cone is invariant under homotheties, i.e. rescaling of the distance function by a positive constant. See Proposition \ref{prop:homothety}.
\end{itemize}
Let us also remark that, in a general metric space, we cannot expect the angle $\bold{\angle_{pxp}}$ to be symmetric in $p$ and $q$. Indeed, already in a smooth Finsler manifold, the symmetry holds if and only if the manifold is actually Riemannian (see for instance the discussion in \cite{G13}-Appendix B).

After this first part on general metric spaces, in Section \ref{sec:(X,d,m)} we consider the case of metric measure spaces, i.e. metric spaces endowed with a positive Borel measure $\mm$; see \eqref{eq:mms} for the precise definition. 
\\In Example \ref{ex:angleRiem} we saw that, already in the smooth case, fixed $p,q \in M$, the best we can hope for in general is that the angle cone $\bold{\angle_{pxq}}$ is single valued for a.e. $x\in X$. In Section  \ref{sec:(X,d,m)} we will see that in a large class of metric measure spaces, namely the Lipschitz-infinitesimally strictly convex and the Lipschitz-infinitesimally Hilbertian ones (see Definition \ref{def:Infin} and Remarks \ref{rem:quadratic}-\ref{rem:Equiv}), it is still the case: fixed $p,q \in X$ the angle cone $\bold{\angle_{pxq}}$ is single valued for $\mm$-a.e. $x\in X$ (see Theorem \ref{thm:Infin}). Moreover, in the Lipschitz-infinitesimally Hilbertian spaces, the angle $\bold{\angle_{pxq}}$ is symmetric in $p$ and $q$, for $\mm$-a.e. $x \in X$ and every $p,q \in X$ (see Theorem \ref{thm:Infin}).  This class of spaces includes many interesting examples: smooth Riemannian  manifolds, finite dimensional Alexandrov spaces with curvature bounded from below, Gromov-Hausdorff limits of Riemannian manifolds with Ricci curvature bounded from below (for more details see Remark \ref{rem:ExIHS}).   

Finally, in Subsection \ref{Subsec:Honda}, we analyze the case  the metric measure space $(X,\sfd,\mm)$ is a Gromov-Hausdorff limit of a sequence of Riemannian manifolds with Ricci curvature bounded below. In this framework, suitably adapting Alexandrov's formula for the angles \eqref{eq:angleAlex}, Honda \cite{Honda12} defined the angle between a.e. triple of points. More precisely,  fixed $p,q\in X$, for $\mm$-a.e. $x \in X$ he defined a unique angle $\angle(p,x,q)$ between  $p$ and $q$ at $x$ as the angle formed by the geodesics from $p$ and $q$ to $x$. It is essential to observe that the construction of Honda holds just for  $\mm$-a.e. $x \in X$; indeed  Colding and Naber  (see Theorem 1.2 and 1.3 in \cite{ColNab}) gave an example of a pointed limit space $(X,\sfd, \mm, \bar{x})$  (which is even a non collapsed limit space such that every point is regular, i.e. the unique tangent cone is a euclidean space) such that for every couple of geodesics $\gamma_1, \gamma_2$ beginning at $\bar{x}$ and every $\theta \in [0,\pi]$ there exists a sequence $t_i \down 0$ such that
$$\cos \theta= \lim_{i \to \infty} \frac{2 t_i^2- \sfd^2(\gamma_1(t_i), \gamma_2(t_i))}{2t_i^2}.$$
%So in limit spaces it is not possible, at every point, to define a unique angle even locally; nevertheless Honda shows that this is possible for a.e. point. 

In Theorem \ref{thm:equivalenceAngle}  we prove that for every fixed $p,q \in X$, the construction of the angle of Honda coincides with our definition $\mm$-a.e.:
$$\bold{\angle_{pxq}}=\{\angle(p,x,q)\}$$ 
for $\mm$-a.e. $x \in X$.  In order to prove the angles equivalence Theorem \ref{thm:equivalenceAngle}, in Lemma \ref{lem:<>=Dnabla}, we show the $\mm$-a.e. equivalence of the infinitesimal Hilbertian structures of the limit spaces in the sense of Cheeger-Colding \cite{ChCo3} and in the sense of Definition \ref{def:Infin}.

\smallskip

\noindent {\bf Acknowledgment.}  The author acknowledges the support
of the ERC ADG GeMeThNES. He  wishes to express his deep gratitude to Luigi Ambrosio for inspiring discussions on metric spaces and to Nicola Gigli for his comments on  the final manuscript.

\section{A notion of angle cone in general metric spaces}\label{sec:(X,d)}
\subsection{Preliminaries and notations}\label{Subsec:Prel}
In this fist part of the paper  $(X,\sfd)$ is a general metric space, later in the paper we will assume more structure; the open ball of center $x \in X$ and radius $r>0$ is denoted with $B_r(x)$. 
\\Given two points $x_0,x_1 \in X$, a continuous curve $\gamma:[0,1]\to X$ satisfying $\gamma(0)=x_0$, $\gamma(1)=x_1$ is said to be a (constant speed, length minimizing) geodesic if 
\begin{equation}\label{defgeo}
\sfd(\gamma(s),\gamma(t))=|t-s|\, \sfd(\gamma(0),\gamma(1))\qquad\forall s,\,t\in
[0,1].
\end{equation}
We say that $(X,\sfd)$ is a \emph{geodesic space} if every pair of points $x_0,x_1 \in X$ is joined by a geodesic; more generally we say that $(X,\sfd)$ is a \emph{length space} if for every pair of points $x_0,x_1 \in X$ the infimum of the lengths of continuous curves joining $x_0$ to $x_1$ is equal to $\sfd(x_0,x_1)$. 
%We call $(X,\sfd)$ \emph{a geodesic space} if for any $x_0,\,x_1\in
%X$ there exists a geodesic joining them. We will denote by $\geo(X)$ the space of all constant speed
%geodesics $\gamma:[0,1]\to X$, namely $\gamma\in\geo(X)$ if
%\eqref{defgeo} holds. 
%Recall also that the weaker notion of \emph{length} space: 
%for all $x_0,\,x_1\in X$ and $\eps>0$ there exists $\gamma\in AC([0,1];X)$ such
%that $\int_0^1|\dot\gamma_t|\,\d t<\sfd(x_0,x_1)+\eps$.

The space of real valued Lipschitz functions on $(X,\sfd)$ is denoted with $\LIP(X)$.
For a  function $f:X \to \R$ the \emph{slope}, also called \emph{local Lipschitz constant}, $|Df|:X \to [0,+\infty]$ is defined by
\begin{equation}\label{eq:slope}
|Df|(x_0):=\limsup_{y\to x}\frac{|f(y)-f(x)|}{\sfd(y,x)} \text{ if $x_0$ is not isolated,}\quad  0 \text{ otherwise}.
\end{equation}
For $f,g:X \to \R$ it is easy to check that
\begin{eqnarray}
|D(\alpha f+ \beta g)|&\leq& |\alpha|\, |Df|+ |\beta|\, |Dg| \quad \forall \alpha, \beta \in \R,\label{eq:slopeSubAdd} \\
|D(fg)|&\leq& |f|\, |Dg|+ |g| \, |Df|. \label{eq:slopeProduct}
\end{eqnarray}
Fixed  $f,g:X \to \R$,  observe that \eqref{eq:slopeSubAdd} implies the convexity of the map $\varepsilon \mapsto |D(g+\varepsilon f)|$:
\begin{equation}\label{eq:ConvexSlope}
|D(g+(\lambda \varepsilon_0+(1-\lambda) \varepsilon_1)f)|\leq \lambda |D(g+\varepsilon_0 f)|+(1-\lambda) |D(g+\varepsilon_1 f)|,
\end{equation}
for any $\lambda \in [0,1]$, $\varepsilon_0,\varepsilon_1 \in \R$.

As noticed by N. Gigli  (see Section 3.1 in \cite{G12}; see also \cite{GM12}) in the slightly different framework of weak upper gradients,  the convexity property \eqref{eq:ConvexSlope} is the key for defining the duality $Df(\nabla g)$ between the differential of $f$ and the gradient of $g$ (notice that in a metric space there is no notion of differential or gradient since the notion of tangent vector space is missing; we have just a notion of \textquotedblleft modulus of the differential of a function\textquotedblright, the point  is to exploit just this concept in order to define \textquotedblleft the differential of $f$ acting on the gradient of $g$\textquotedblright). The idea is the following: notice that if $\varphi:\R \to \R^+$  is a convex function, the values of 
$$\liminf_{\varepsilon \down 0} \frac{\varphi(\varepsilon)^p-\varphi(0)^p}{p \varepsilon \varphi(0)^{p-2}}, \quad  \limsup_{\varepsilon \up 0} \frac{\varphi(\varepsilon)^p-\varphi(0)^p}{p \varepsilon \varphi(0)^{p-2}}$$
are independent of $p \in (1,+\infty)$, and equal to $\varphi(0) \varphi'(0^+)$, $\varphi(0) \varphi'(0^-)$ respectively, as soon as $\varphi(0)\neq 0$. The convexity also grants that the limits exist and can be substituted by $\inf_{\varepsilon>0}$ and $\sup_{\varepsilon<0}$, respectively. Therefore we can give the following definition.

\begin{definition}[$D^{\pm} f(\nabla g)$] \label{def:Dfnablag}
Let $(X,\sfd)$ be a metric space and $f,g: X \to \R$  with finite slopes $|Df|,|Dg|: X \to [0,+\infty)$. The functions $D^{\pm}f (\nabla g):X \to \R$ are defined by 
\begin{eqnarray}
D^{+} f (\nabla g)&:=& \liminf_{\varepsilon \downarrow 0} \frac{|D(g+\varepsilon f)|^2-|D g|^2}{2 \varepsilon}, \label{def:D+fnablag}\\
D^{-} f (\nabla g)&:=& \limsup_{\varepsilon \uparrow 0} \frac{|D(g+\varepsilon f)|^2-|D g|^2}{2 \varepsilon}, \label{def:D-fnablag}
\end{eqnarray}
on $\{|D g| \neq 0 \}$, and are taken $0$ by definition on $\{|D g| =0 \}$.
\end{definition}

We remark  that the definition of $D^{\pm}f(\nabla g)$ is \emph{pointwise} while the analogous objects defined in \cite{G12} and \cite{GM12} using weak upper gradients were defined almost everywhere with respect to a reference measure (this was unavoidable since already the weak upper gradient is an object defined almost everywhere). Fixed  a point $x_0 \in X$, notice that for  the definitions \eqref{def:D+fnablag} and \eqref{def:D-fnablag} of  $D^{\pm} f (\nabla g)(x_0)$ to make sense, it is sufficient that $|D f|$ and $|D g|$ are finite just at the point $x_0$.

\begin{remark}
\rm{
Thanks to the discussion before Definition \eqref{def:Dfnablag}, it is clear that for every $p\in (1,\infty)$  we can equivalently set
\begin{eqnarray}
D^{+} f (\nabla g)&=& \liminf_{\varepsilon \downarrow 0} \frac{|D(g+\varepsilon f)|^p-|D g|^p}{p \varepsilon |D g|^{p-2}}, \label{def:D+fnablagp}\\
D^{-} f (\nabla g)&=& \limsup_{\varepsilon \uparrow 0} \frac{|D(g+\varepsilon f)|^p-|D g|^p}{p \varepsilon |D g|^{p-2}}. \label{def:D-fnablagp}
\end{eqnarray}
Moreover the limits in the definitions exist and can be substituted by $\inf_{\varepsilon>0}$ ($\sup_{\varepsilon<0}$ respectively}
\fr
\end{remark}

Notice that the inequality $|D(g+\varepsilon f)|\leq |D g|+|\varepsilon |\,|D f|$ yields
\begin{equation}\label{Schwartz}
|D^{\pm}f(\nabla g)| \leq |D f|\;|D g|,
\end{equation}
and in particular  if $f,g$ are Lipschitz (resp. locally Lipschitz) then $D^{\pm}f(\nabla g)$ is bounded (resp.  locally bounded). Moreover, the convexity of $f \mapsto |D(g+f)|^2$ gives 
\begin{equation}\label{eq:D-<D+}
D^-f(\nabla g) \leq D^+ f (\nabla g).
\end{equation}
Also, from the definition it directly follows that
\begin{equation}\label{eq:sign}
D^\pm  (-f)(\nabla g)=D^\pm f(\nabla(-g)), \quad D^\pm  (-f)(\nabla g)=-D^{\mp}f(\nabla g),
\end{equation}
and that
\begin{equation}\label{eq:Dg2}
D^{\pm}g(\nabla g)=|D g|^2.
\end{equation}

\subsection{The definition of angle cone and some examples}\label{Subsec:DefAngleCone}
Given a metric space $(X,\sfd)$ and two points  $p,q \in X$, consider the distance functions 
\begin{equation} \label{eq:rprq}
r_p(\cdot):=\sfd(p,\cdot),\quad r_q(\cdot):=\sfd(q,\cdot).
\end{equation}
Since both $r_p$ and $r_q$ are $1$-Lipschitz on $X$, recalling Definition \ref{def:Dfnablag}, for every $x \in X$ we can define
\begin{equation}\label{eq:cpm}
c^+_{pxq}:=D^+ r_p(\nabla r_q) (x), \quad c^-_{pxq}:= D^- r_p(\nabla r_q)(x). 
\end{equation}
The estimate \eqref{Schwartz} ensures that $c^+_{pxq}, c^-_{pxq} \in [-1,1]$, hence, as the function $\arccos: [-1,1]\to [0, \pi]$ is strictly decreasing, it makes sense to set
\begin{equation}\label{eq:angle}
\angle^-_{pxq}:= \arccos c^+_{pxq},  \angle^+_{pxq}:= \arccos c^-_{pxq},
\end{equation}
and to define the angle cone between $p,x,q$ as follows.
\begin{definition}\label{def:anglecone(X,d)}
Let $(X,\sfd)$ be a metric space and consider a triple $(p,x,q) \in X^3$. The \emph{angle cone} of  $(p,x,q)$, denoted with $\bold{\angle_{pxq}}$, is the interval
\begin{equation}\label{eq:anglecone(X,d)}
\bold{\angle_{pxq}}:= [\angle^-_{pxq}, \angle^+_{pxq}] \subset [0,\pi],
\end{equation}  
where $\angle^-_{pxq}, \angle^+_{pxq}$ are defined above in \eqref{eq:angle}.
\end{definition}

\begin{remark} %[The angles $\bold{\angle_{ppq}}$ and $\bold{\angle_{pqq}}$]
\rm{
The case $x=p$ or $x=q$ is degenerate. Since the geometrically interesting case is $x \neq p,q$, we will always make this assumption in the rest of the paper.
%; indeed directly from the definition of slope \eqref{eq:slope} and from Definition \ref{def:Dfnablag} it follows that
%\begin{eqnarray}
%-|Dr_q|(p)\leq D^+r_p(\nabla r_q) (p)  & \& &   D^-r_p(\nabla r_q) (p) \leq |Dr_q|(p) \nonumber \\
%-|Dr_p|(q)\leq D^-r_p(\nabla r_q) (q)  & \& &   D^+r_p(\nabla r_q) (q) \leq |Dr_p|(q) \nonumber.
%\end{eqnarray} 
%Therefore
%\begin{eqnarray}
%\bold{\angle_{ppq}}&\subset&\big[\arccos(-|Dr_q|(p)), \arccos (|Dr_q|(p))\big], \nonumber \\
%\bold{\angle_{pqq}}&\supset&\big[\arccos(|Dr_p|(q)), \arccos (-|Dr_p|(q))\big]. \nonumber
%\end{eqnarray}
}
\fr
\end{remark}
In order to get familiar with the definition, let us discuss immediately some examples which also show that this notion of angle cone formalizes what the intuition suggests. 

\begin{example}[In a Riemannian manifold the angle is single valued outside the cut locus] \label{ex:Mg}
{\rm  Given a Riemannian $n$-manifold $(M^n,g)$ and $p,x,q \in M$ such that $x \neq p,q$ is not in the cut locus of $p$ and $q$, it is clear that the distance functions $r_p:=\sfd(p,\cdot)$ and $r_q:=\sfd(q,\cdot)$ are smooth in a neighborhood of $x$ and it  is easy to check that 
\begin{equation} \label{eq:angleMg}
\cos(\angle^-_{pxq})=\cos(\angle^+_{pxq})=g_x(\nabla r_p(x), \nabla r_q (x)),
\end{equation}
$\nabla r_p(x), \nabla r_q (x)$ being the gradients of $r_p,r_q$ evaluated at  $x$. Since the cut locus has null measure with respect to the Riemannian volume measure (more precisely it is $(n-1)$ rectifiable, see \cite{MantMenn}), it follows that given $p,q \in X$ the angle between $p,q$ and $x$ is  single valued for a.e.-$x \in M$ and his value is given by \eqref{eq:angleMg}.

Observe that if $(M,g)$ is a complete, simply connected Riemannian manifold with non positive sectional curvature   (i.e. $(M,g)$ is an Hadamard manifold), then by Hadamard Theorem (see for instance Theorem 3.1 in \cite{DoC}) the cut locus of every point is empty; therefore in this case the angle of any triple of points $p,x,q \in M$ is single valued and given by formula  \eqref{eq:angleMg}.
% \nota Pensare al caso di Hadamard spaces (articolo Lurie e libro su Non positively curved spaces pag. 212 teorema di Hadamard) \fn
} \fr
\end{example}

As sketched intuitively in the Introduction, in general the angle between three points can be multivalued even in the case the metric space $(X,\sfd)$ is a smooth Riemannian manifold, if $x$ is in the cut locus of $p$ or $q$. Now, thanks to the introduced concept of angle cone, we can make rigorous this intuition.

\begin{example}[Non uniqueness of the angle caused by the cut locus]\label{Ex:S2}
{ \rm 
Let $(X,\sfd)$ be the standard 2-dimensional sphere $S^2$ of unit radius centered in the origin of $\R^3$   with the distance induced by the usual round Riemannian metric. Let $x=(0,0,1)$ be the north pole, $q=(0,0,-1)$ be the south pole and $p=(1,0,0)$. Since $x$ is in the cut locus of $q$, it is natural to expect that the angle between $p,x,q$ is not single valued; more precisely, since any half great circle joining $x$ and $q$ is a length minimizing geodesic and since the gradient of $r_q$ is the tangent vector to the minimizing geodesic, then intuitively all the tangent space $T_xS^2$ is spanned by the different gradients of $r_q$ coming from all the different minimizing geodesics. Therefore it is natural to expect that \emph{any} angle is in the angle cone of $pxq$. 

Using the introduced notions we can easily make rigorous  this intuition: a straightforward computation using just the definition of slope \eqref{eq:slope} and Definition \ref{def:Dfnablag}  yields (for computing the first quantity consider  variations in the direction from $x$ towards $p$, for the second one consider  variations in the opposite direction)
$$D^+r_p(\nabla r_q)=1, \quad D^-r_p(\nabla r_q)=-1;$$
hence  definitions \eqref {eq:angle} and \eqref{eq:anglecone(X,d)} give that the angle cone $\bold{\angle_{pxq}}$ coincides with $[0,\pi]$, confirming the intuition. } \fr
\end{example} 

The second cause for the non uniqueness of the angle is the  potential lack of strict convexity of the space, this is best understood with an example.

\begin{example}[Non uniqueness of the angle caused by the non strict convexity of the  norm]\label{Ex:RdueLinfty}
{ \rm 
Let $(X,\sfd)$  be $\R^2$ endowed  with the $L^\infty$ norm on its coordinates and take $p=(1,0)$, $q=(0,1)$ and $x=(1,1)$. Notice that any vector of the form $(v_1,1), v_1\in [-1,1]$, can be called gradient  of $r_p$ at $x$. Indeed all of them have unit norm and the derivative of $r_p$ at $x$ along any of them is $1$, $1$ being also the dual norm of the differential of $r_p$. Analogously, all the vectors of the form $(1,v_2), v_2 \in [-1,1]$, can be called gradient of $r_q$ at $x$. Taking the (euclidean) angles between all the possible pairs $(v_1,1),(1,v_2)$ with $v_1,v_2 \in [-1,1]$, we span all the interval $[0,\pi]$; it is therefore natural to expect that the angle cone $\bold{\angle_{pxq}}$ is the whole $[0,\pi]$.

This is confirmed by a straightforward computation: taking variations in the direction of the vector $(1,1)$, respectively $(1,-1)$, we obtain using just the definition of slope \eqref{eq:slope} and Definition \ref{def:Dfnablag} that
$$D^+r_p(\nabla r_q)=1, \quad D^-r_p(\nabla r_q)=-1;$$
hence  definitions \eqref {eq:angle} and \eqref{eq:anglecone(X,d)} give that the angle cone $\bold{\angle_{pxq}}$ coincides with $[0,\pi]$ confirming the intuition. } \fr
\end{example} 

Since in both the examples the angle cone is completely degenerate, i.e. it is the whole $[0,\pi]$, it is a natural question if this is always case when non uniqueness occurs. This is false in both the situations, i.e. non uniqueness by cut locus and non uniqueness by lack of strict convexity. In case of Example \ref{Ex:S2} it is enough to replace the whole sphere $S^2$ by a sector containing $p=(1,0,0)$ delimited by two half great circles  having $x=(0,0,1), q=(0,0,-1)$ as endpoints; in this case the angle cone depends on the amplitude and on the position of the sector with respect to $p=(1,0,0)$. In case of Example \ref{Ex:RdueLinfty}, keeping fixed $p=(1,0)$ and $q=(0,1)$ and varying $x\in \R^2$ it is easy to see that also the angle cone varies.  

\subsection{Some properties of the angle cone}\label{Subsec:Properties}
In this subsection we prove two basic properties of the angle cone: the first says that for every $p,x \in X$, the angle cone $\bold{\angle_{pxp}}$ is single valued, and if moreover $(X,\sfd)$ is a length space then the angle is actually null; the second says that the angle cone is invariant under homotheties of the metric. In order to prove the first property we need a lemma concerning the following property of length spaces: while in a general metric space, using the triangle inequality in the definition of slope \eqref{eq:slope}, one has just $|D r_p|\leq 1$, in length spaces actually it holds has $|D r_p|\equiv 1$.

\begin{lemma}\label{lem:Dr}
Let $(X,\sfd)$ be a metric space,  fix $p \in X$ and consider $r_p(\cdot):=\sfd(p,\cdot)$. Then  
$$|Dr_p|\leq 1\quad \text{on  } X.$$
In case $(X,\sfd)$ is a length space, then 
$$|Dr_p|\equiv 1\quad \text{on  } X.$$
\end{lemma}

\begin{proof}
The first part of the lemma is a straightforward application of the triangle inequality in the definition \eqref{eq:slope} of the slope.

Let us first give the proof of the second claim in the easy case when $(X,\sfd)$ is a geodesic space; the idea in the more general case of length spaces is the same, the technique being a bit more involved. Let $\gamma:[0,1]\to X$ be a length minimizing geodesic  from $p$ to $x$, then 
$$
|D r_p|(x)=\limsup_{y \to x} \frac{|\sfd(p,x)-\sfd(p,y)|}{\sfd(x,y)}\geq \limsup_{t \to 1} \frac{|\sfd(p,x)-\sfd(p,\gamma(t))|}{\sfd(x,\gamma(t))}=\limsup_{t \to 1} \frac{\sfd(x,\gamma(t))}{\sfd(x,\gamma(t))}=1,
$$
where the first equality is just the definition of slope \eqref{eq:slope} applied to  the function $r_p(\cdot):=\sfd(p,\cdot)$, the inequality comes from the fact that by construction $\lim_{t \to 1}\gamma(t)=x$, and finally we used that every segment of a length minimizing geodesic minimizes the length. Hence, in this case, $|D r_p|(x)=1$.
 
Now we prove the lemma in case $(X,\sfd)$ is a length space.  Let $\gamma_n\in C([0,1], X)$ be such that $\gamma_n(0)=p$, $\gamma_n(1)=x$ and $\length(\gamma_n)\leq \sfd(p,x)+\frac{1}{n}$.
Then  
\begin{eqnarray}
|D r_p|(x)&:=& \limsup_{y\to x}  \frac{|\sfd(p,x)-\sfd(p,y)|}{\sfd(x,y)}\geq \limsup_{n\to \infty} \limsup_{t \to 1} \frac{|\sfd(p,x)-\sfd(p,\gamma_n(t))|}{\sfd(x,\gamma_n(t))} \nonumber \\
&=& \lim_{k\to \infty} \lim_{i\to \infty}  \frac{|\sfd(p,x)-\sfd(p,\gamma_{n_k}(t_{i,k}))|}{\sfd(x,\gamma_{n_k}(t_{i,k}))}; \label{eq:Dr1}
\end{eqnarray}
where, in the last equality we chose sequences $(n_k)$ and $(t_{i,k})$ realizing the $\limsup$s. Now let us choose a subsequence $i_k$ such that for every $k$ the following hold:
\begin{eqnarray}
\sfd(x, \gamma_{n_k}(t_{i_k,k}))&\leq& \frac{1}{\sqrt{n_k}} , \label{eq:sqrtnk} \\
\left(\lim_{i\to \infty} \frac{|\sfd(p,x)-\sfd(p,\gamma_{n_k}(t_{i,k}))|} {\sfd(x,\gamma_{n_k}(t_{i,k}))} \right) &-&  \frac{|\sfd(p,x)-\sfd(p,\gamma_{n_k}(t_{i_k,k}))|} {\sfd(x,\gamma_{n_k}(t_{i_k,k}))} \leq \frac{1}{n_k}. \nonumber
\end{eqnarray}  
Therefore we can write \eqref{eq:Dr1} as
\begin{eqnarray} 
|D r_p|(x)&\geq&  \lim_{k\to \infty}  \frac{|\sfd(p,x)-\sfd(p,\gamma_{n_k}(t_{i_k,k}))|}{\sfd(x,\gamma_{n_k}(t_{i_k,k}))} \nonumber \\
          &\geq&  \limsup_{k \to \infty} \frac{\length(\gamma_{n_k})-\frac{1}{n_k}-\length\big({{\gamma_{n_k}}_|}_{[0,t_{i_k,k}]}\big)}{\sfd(x,\gamma_{n_k}(t_{i_k,k}))}; \label{eq:Dr2}
\end{eqnarray} 
where, in the last estimate,  we used that by construction $\sfd(p,x) \geq \length(\gamma_{n_k})-\frac{1}{n_k}$ and the trivial inequality $\sfd(p,\gamma_{n_k}(t_{i_k,k})) \leq \length\big({{\gamma_{n_k}}_|}_{[0,t_{i_k,k}]}\big)$. To conclude, observe that
\begin{equation}\nonumber
\length(\gamma_{n_k}) -\length\big({{\gamma_{n_k}}_|}_{[0,t_{i_k,k}]}\big)= \length\big({{\gamma_{n_k}}_|}_{[t_{i_k,k},1]}\big) \geq \sfd(\gamma_{n_k}(t_{i_k,k}), x);
\end{equation}	
which, plugged in \eqref{eq:Dr2} together with \eqref{eq:sqrtnk}, gives
$$|D r_p|(x)\geq \limsup_{k\to \infty} \frac{ \sfd(\gamma_{n_k}(t_{i_k,k}), x)-\frac{1}{n_k}}{\sfd(\gamma_{n_k}(t_{i_k,k}), x)} \geq 1-\lim_{k\to \infty} \frac{1/n_k}{1/{\sqrt{n_k}}}=1.$$
\end{proof}

\begin{proposition}[In length spaces $\bold{\angle_{pxp}}=\{0\}$]\label{prop:pxp=0}
Let $(X,\sfd)$ be a  metric space. Then for any $p,x \in X$ the angle cone $\bold{\angle_{pxp}}$ reduces to a single angle: $\bold{\angle_{pxp}}=\{\theta\}\subset[0,\pi]$. If moreover $(X,\sfd)$ is a length space then  $\bold{\angle_{pxp}}=\{0\}$
\end{proposition}

\begin{proof}
Recalling the definition of $c^{\pm}_{pxq}$ in \eqref{eq:cpm} and the property \eqref{eq:Dg2}, we have
\begin{equation}\label{eq:cp=cm}
c^+_{pxp}=c^-_{pxp}=|D r_p|^2(x).
\end{equation}
By Lemma \ref{lem:Dr}, if $(X,\sfd)$ is a general metric space one has just $|D r_p|\leq 1$; but if $(X,\sfd)$ is a length space then $|D r_p|(x)=1$. The conclusion follows recalling that from the definition of angle cone \eqref{eq:anglecone(X,d)} (see also \eqref{eq:angle}) one has
$${\bold{\angle_{pxp}}}=\{\arccos (|D r_p|^2(x))\}.$$
\end{proof}

\begin{remark}\label{rem:pxp=0}
{\rm  Proposition \ref{prop:pxp=0} suggests that the natural class of  metric spaces for which our definition of angle cone has a geometric relevance are the length spaces. 
%\item For proving that $\bold{\angle_{pxp}}=\{0\}$ actually we just used that $|D r_p|(x)=1$, fact which is ensured if just $\sfd(x,p)$ is the infimum of the length of continuous curves joining $x$ and $p$. Hence, strictly speaking, the assumption that $(X,\sfd)$ is a length space is not necessary for saying that $\bold{\angle_{pxp}}=\{0\}$ for fixed $p,x \in X$. 
} \fr 
\end{remark}
\medskip

A desirable property of the angles is the invariance under homothety, i.e. under constant rescaling of the metric; this is the content of the next proposition.

\begin{proposition}[Invariance of the angle cone under constant rescaling of the metric]\label{prop:homothety}
Let $(X,\sfd)$ be a metric space, fix $\lambda>0$ and consider the new metric $\sfd_\lambda(\cdot,\cdot):= \lambda \sfd(\cdot,\cdot)$ on $X$. Then for every triple of points $p,q,x\in X$, the angle cones evaluated with respect the two different metrics $\sfd(\cdot,\cdot)$ and $\sfd_{\lambda}(\cdot,\cdot)$ coincide:
$$
\bold{\angle_{pxq}}=\bold{\angle^\lambda_{pxq}},
$$
where $\bold{\angle^\lambda_{pxq}}$ is the angle cone of the triple $(p,x,q)\in X^3$ computed in the metric space $(X,\sfd_\lambda)$.
\end{proposition}

\begin{proof}
Called $r^\lambda_p(\cdot):=\sfd_\lambda(p,\cdot)$ and $|D r^\lambda_p|_\lambda$ the slope of $r^\lambda_p$ in the metric space $(X,\sfd_\lambda)$, directly from the definition of slope \eqref{eq:slope}, for any $x \in X$ we get  
$$
|D r^\lambda_p|_\lambda(x)=\limsup_{y\to x} \frac{|r^\lambda_p(x)-r^\lambda_p(y)|}{\sfd_\lambda(x,y)}= \limsup_{y\to x} \frac{\lambda |r_p(x)-r_p(y)|}{\lambda \, \sfd(x,y)}=|D r_p|(x).
$$
Since the definition of angle cone \eqref{eq:anglecone(X,d)} (recall also Definition \ref{def:Dfnablag}, \eqref{eq:cpm} and \eqref{eq:angle}) involves just the slopes of $r_p, r_q$, and $r^\lambda_p, r^\lambda_q$ respectively, the conclusion follows.  
\end{proof}

\section{The notion of angle in metric measure spaces}\label{sec:(X,d,m)}

Let us start this section by some motivations coming from the smooth case. Let  $(M,g)$ be a smooth Riemannian manifold and  $p,x,q \in M$, let us denote with $C_p,C_q \subset M$ the cut loci of $p$ and $q$ respectively and with $\mm_g$ the Riemannian volume measure associated to $g$. As discussed in the previous examples \ref{ex:Mg} and \ref{Ex:S2} the following holds:
\begin{itemize}
\item if $x \notin(C_p \cup C_q \cup \{p,q\})$  then the angle cone $\bold{\angle_{pxq}}$ is single valued,
\item if $x \in C_p \cup C_q $  we cannot expect the angle cone $\bold{\angle_{pxq}}$ to be single valued,
\item $\mm_g(C_p\cup C_q \cup \{p,q\})=0$.
\end{itemize}
Therefore, fixed $p$ and $q$ in $M$, the best we can expect is that the angle cone  $\bold{\angle_{pxq}}$ is \emph{single valued for a.e. $x \in M$}.
\\

This observation suggests that, for studying the uniqueness of the angle among three points, it is natural to endow the metric space $(X,\sfd)$ with a measure $\mm$, obtaining  the  \emph{metric measure space}, $\mms$ space for short, $(X,\sfd,\mm)$.  In this section  we will always assume that 
\begin{equation}
\label{eq:mms}
\begin{split}
&\quad \quad \quad  \quad (X,\sfd) \text{ is  a complete separable metric space and} \\
&\mm  \text{  is a Borel locally finite non negative measure on } X \text{ with } \supp \mm= X,
\end{split}
\end{equation}
where by local finiteness we mean that for every $x \in X$ there exists a neighborhood   $U_x$ of $x$ such that $\mm(U_x)< \infty$, and $\supp \mm$ denotes the smallest closed subset where $\mm$ is concentrated. Notice that the assumption $\supp \mm= X$ can be always fulfilled by replacing $X$ with $\supp \mm$, but we make it in order to simplify some statements in the sequel. 

The measure $\mm$ is said to be   \emph{doubling} if  for some constant $C>0$ it holds
\begin{equation}\label{eq:doubling}
\mm(B_{2r}(x))\leq C\mm(B_r(x)),\qquad\forall x\in X,\ r>0.
\end{equation}
Let $p_0\geq 1$. We say that $(X,\sfd,\mm)$ supports a weak local $(1,p_0)$-Poincar\'e inequality (or more briefly a $p_0$-Poincar\'e inequality) if there exist constants $C_{PI}$ and $\lambda\geq 1$ such that for all  $x\in X$, $r>0$ and Lipschitz functions $f:X\to\R$ it holds
\begin{equation}\label{LPI}
\frac1{\mm(B_r(x))}\int_{B_r(x)} |f-f_{B_r(x)}| \, \d \mm \leq C_{PI}\, 2r \left(\frac{1}{\mm(B_{\lambda r}(x))} \int_{B_{\lambda r}(x)} |D f|^{p_0} \, \d \mm \right)^{\frac{1}{p_0}},
\end{equation}
where $f_{B_r(x)}:=\frac1{\mm(B_r(x))}\int _{B_r(x)} f \, \d \mm$. Notice that typically the Poincar\'e inequality is required to hold for integrable functions $f$ and upper gradients $G$ (see for instance Definition 4.1 in \cite{Bjorn}), rather than for Lipschitz functions and their slope. Since the slope is an upper gradient, it is obvious that the second formulation implies the one we gave, but also the converse implication holds, as a consequence of the density in energy of Lipschitz functions in the Sobolev spaces proved in \cite{AmbrosioGigliSavare12} for the case $p_0>1$ and in \cite{Ambrosio-DiMarino} for $p_0=1$.  Therefore the definition we chose is equivalent to the standard one. Let us remark that this equivalence was proven earlier in \cite{HeinKos} for proper, quasiconvex and doubling metric measure spaces, while in \cite{Kos} (choosing $X = \R^n \setminus E$ for suitable compact sets $E$) it is proven that completeness of the space cannot be dropped.
\medskip 
%\nota vedere se introdurre anche la Poincar\'e locale, ci serve solo per dire che per funzioni Lipschitz la slope \'e il minimal weak upper gradient (ci serve per dire che le nostre definizioni di inf. strict convexity e inf. Hilbert sono pi\'u deboli di quelle di Nicola). Per\'o forse per la nota di Ambrosio-Colombo-DeMarino per questo basta la doubling. \fn

\subsection{The case of infinitesimally strictly convex and infinitesimally Hilbertian spaces}
The notions of \emph{infinitesimally strictly convex} and \emph{infinitesimally Hilbertian} $\mms$  have been introduced in \cite{G12}  and \cite{Ambrosio-Gigli-Savare11b}  respectively (see also \cite{AGMR12}, \cite{Ambrosio-Gigli-Savare11} and \cite{GM12}) in the  framework of weak upper gradients, here we give a more elementary definition involving just the slope of Lipschitz functions.  In  doubling $\mms$  satisfying a $p_0$-Poincar\'e inequality, for some $p_0\in (1,2)$,  our definitions in terms of the slope are equivalent to the corresponding ones in terms of the weak upper gradients; for a discussion  see Remark \ref{rem:Equiv}.  
In order not to create ambiguity in literature, we call Lipschitz-infinitesimally strictly convex and Lipschitz-infinitesimally Hilbertian the notions in terms of the slope; the precise definition follows.
\begin{definition}\label{def:Infin}
Let $(X,\sfd,\mm)$ be a $\mms$  as in \eqref{eq:mms}. 
\begin{itemize}
\item  We say that $(X,\sfd,\mm)$ is \emph{Lipschitz-infinitesimally strictly convex} if for every $f,g \in \LIP(X)$ it holds
\begin{equation}\label{eq:ISC}
D^+ f(\nabla g)= D^-f (\nabla g) \quad \text{ $\mm$-a.e. in }  X.
\end{equation} 
the objects $D^\pm f(\nabla g)$ being defined in \eqref{def:Dfnablag}.
\item  $(X,\sfd,\mm)$ is said \emph{Lipschitz-infinitesimally Hilbertian} if  for every $f,g \in \LIP(X)$ it holds
\begin{equation}\label{eq:IH}
D^+ f(\nabla g)= D^-f (\nabla g)=D^-g (\nabla f)= D^+ g (\nabla f)\quad \text{$\mm$-a.e. in }   X;
\end{equation} 
in this case the common value above is denoted with $\nabla f \cdot \nabla g$.
\end{itemize}
\end{definition}
Notice that, in particular, Lipschitz-infinitesimally Hilbertian spaces are Lipschitz-infinitesimally strictly convex. Observe also that from  inequality \eqref{eq:D-<D+} it follows that the $\mm$-a.e. equality \eqref{eq:defISC} is equivalent to the integral inequality
\begin{equation}\label{eq:defISC}
\int \big( D^+ f(\nabla g)-D^- f(\nabla g) \big) \, \d \mm \leq 0.
\end{equation}

\begin{remark}[A condition equivalent to  Lipschitz-Inf. Hilbertianity]\label{rem:quadratic}
 {\rm 
Following the arguments in the proof of Proposition 4.20 in \cite{G12} (which in turns is an adaptation of  Section 4.3 of \cite{Ambrosio-Gigli-Savare11b}) it is not difficult to see that $(X,\sfd,\mm)$ is  Lipschitz-infinitesimally Hilbertian if and only if the map $f\mapsto \frac{1}{2} \int |Df|^2 \,\d \mm$ from $\{f \in \LIP(X): |D f| \in L^2(X,\mm)\}$ to $\R^+$ is quadratic, which in turn is equivalent to ask that such a map satisfies the parallelogram identity. \rm} \fr
\end{remark}

The following remark was suggested by Luigi Ambrosio and Nicola Gigli, let me express my gratitude to them.

\begin{remark}[Relationship with the notions of inf. strict convexity and inf. Hilbertianity already present in literature]\label{rem:Equiv}
{\rm 
 As mentioned above, the notions of infinitesimally strictly convex and infinitesimally Hilbertian $\mms$  have been introduced in \cite{G12} and  \cite{Ambrosio-Gigli-Savare11b} respectively,  using the  language  of weak upper gradients; the definition we propose here involves just the more elementary concept of  slope of a Lipschitz function. Here we briefly comment on the relationship between the two approaches; we are not introducing all the terminology in order to keep the presentation short, the interested reader is referred to the aforementioned papers.

\underline{Inf. Hilbertianity}. For a general $\mms$ space $(X,\sfd,\mm)$ as in \eqref{eq:mms}, our definition of  Lipschitz-infinitesimal Hilbertianity  is slightly stronger than the corresponding one in   \cite{Ambrosio-Gigli-Savare11b}-\cite{G12} in terms of weak upper gradients. This is because the two definitions are equivalent to ask that the squared $L^2(X,\mm)$ norm of the slope (resp. of the weak upper gradient), called Cheeger energy, satisfy the  parallelogram identity (see Remark \ref{rem:quadratic} above, Definition 4.18 and Proposition 4.20 in \cite{G12}). Since the Cheeger energy in terms of weak upper gradients is the $L^2(X,\mm)$-lower semicontinuous relaxation  of the corresponding one in terms of the slopes (see Section 4 in \cite{Ambrosio-Gigli-Savare11}), it is not difficult to see that the parallelogram identity at the level of  slopes passes to the limit to the level of weak upper gradients. Then our definition using slopes implies the one in terms of weak upper gradients. 

The converse implication in general is not clear, but in case  $(X,\sfd,\mm)$ is a doubling $\mms$ space as in \eqref{eq:mms}, \eqref{eq:doubling} satisfying a $p_0$-Poincar\'e inequality \eqref{LPI} for some $p_0 \in (1,2)$, then the definitions are equivalent. 
%Let us briefly discuss the relation between the two approaches in case  $(X,\sfd,\mm)$ is a doubling $\mms$ space as in \eqref{eq:mms} \nota (satisfying a local Poincar\'e inequality??serve??)\fn. We are not introducing all the terminology in order to keep the presentation short, the interested reader is referred to the aforemenioned papers.
%The two approached coincide in  doubling $\mms$ spaces \nota (satisfying a local Poincar\'e inequality??serve??)\fn, let us briefly sketch the argument of the equivalence.
This is a straightforward consequence of a celebrated result of Cheeger (Theorem 6.1 in  \cite{Cheeger00}, see also Theorem A.7 in \cite{Bjorn} and the recent revisitation of this part Cheeger's paper given in \cite{ACD}) telling that, for locally Lipschitz functions, the slope coincides with the weak upper gradient $\mm$-a.e.. 
%therefore, if $f,g \in \LIP(X)$, the objects $D^{\pm}f (\nabla g)$ defined here in \eqref{def:Dfnablag} in terms of the slope, and the corresponding quantities defined by Gigli (Definition 3.1 in \cite{G12}; see also \cite{GM12})  in terms of the weak upper gradients coincide $\mm$-a.e.. It is therefore clear that the notions of inf. strict convexity and of inf. Hilbertianity expressed in terms of weak upper gradients imply the corresponding ones in terms of slopes of Lipschitz functions, i.e. the definitions of Gigli are stronger than ours.

\underline{Inf. strict convexity}.  In case $(X,\sfd,\mm)$ is a doubling $\mms$  satisfying a $p_0$-Poincar\'e inequality, $p_0\in(1,2)$, for the same argument above the definition in terms of weak upper gradients implies our definition in terms of slope of Lipschitz functions; the converse  implication, again in  doubling\&Poincar\'e spaces, follows by the aforementioned Theorem of Cheeger together with the Lusin approximation of Sobolev functions via locally Lipschitz functions (for the proof see for instance Theorem 5.1 in \cite{Bjorn}; see also Theorem 2.5 in \cite{GM12} for the precise statement we are referring to). Both the implications are open in  general $\mms$.
\\

Recall that measured-Gromov-Hausdorff limits of Riemannian manifolds with Ricci curvature bounded from below, and more generally the metric spaces satisfying the Curvature Dimension condition $CD(K,N)$ for some $K\in \R, N\in \R^+$, satisfy the aforementioned doubling and  Poincar\'e conditions on bounded subsets (see \cite{ChCo3}, Corollary 2.4 in \cite{Sturm06II}  and Theorem 1.2 in \cite{RajalaCalcVar}); hence on such spaces, the definition of infinitesimal strict convexity  and  of infinitesimal Hilbertianity in terms of the slope of Lipschitz functions is equivalent to the corresponding one in terms of weak upper gradients. 
%Regarding the infinitesimal Hilbertianity,  the equivalence between our  definition and the one using weak upper gradients is an easy consequence of the following facts:
%\begin{itemize}
%\item If $(X,\sfd, \mm)$ is a doubling \& $p$-Poincar\'e space, $p>1$; then for every $f \in S^p_{loc}(X,\sfd,\mm)$ there exists a sequence $(f_n)$ of Lipschitz functions such that $|D(f_n-f)|_w \to 0 $ in $L^p_{loc}(X,\mm)$ and therefore, up to subsequences, $|D(f_n-f)|_w \to 0$ $\mm$-a.e. (for the existence of the approximation see Theorems 2.5 and 2.6 in \cite{GM12}, for the proofs see for instance Theorem 5.1 in \cite{Bjorn});
%\item As remarked above, under the current assumptions on $(X,\sfd,\mm)$, if $f,g \in \LIP(X)$ then the quantities $D^{\pm}f (\nabla g)$ and the corresponding ones in terms of weak upper gradients coincide $\mm$-a.e.;
%\item For any $f_1,f_2,g \in S^p_{loc}(X,\sfd,\mm)$ the following Lipschitz estimate holds (see Proposition 3.2 in \cite{G12})
%$$|(D^{\pm} f_1 \nabla g)-(D^{\pm} f_2 \nabla g)|\leq |D (f_1-f_2)|_w\; |Dg|_w  \quad \mm\text{-a.e.}$$
%\end{itemize}
} \fr
\end{remark}

\begin{remark}[Relevant examples of  Lipschitz-inf. strictly convex and Lipschitz-inf. Hilbertian spaces]\label{rem:ExIHS}
{\rm 

Many important classes of metric spaces are Lipschitz-inf. strictly  convex or  Lipschitz-inf. Hilbertian: 
\begin{itemize}

\item Smooth Riemannian manifolds are Lipschitz-infinitesimally Hilbertian.

\item In \cite{KMS00}, using key results of \cite{OS94} and \cite{Per} it has been proved that finite dimensional Alexandrov spaces with curvature bounded from below are Lipschitz-infinitesimally Hilbertian.

\item Measured Gromov-Hausdorff limits of Riemannian manifolds with Ricci curvature bounded from below are Lipschitz-infinitesimally Hilbertian, see Section 6.3 in \cite{Ambrosio-Gigli-Savare11b} and \cite{GigliMondinoSavare} for a deeper discussion on stability issues.

\item In case of normed spaces, Lipschitz-infinitesimal strict convexity is equivalent to strict convexity of the norm, which in turn is equivalent to the differentiability of the dual norm in the cotangent space. 

\end{itemize}
}\fr
\end{remark}

The following theorem  is a direct consequence of the definition of angle cone \eqref{eq:anglecone(X,d)} which in turn is given by \eqref{eq:cpm} and \eqref{eq:angle}, and of  Definition \ref{def:Infin}.
 
\begin{theorem}\label{thm:Infin}
Let $(X,\sfd,\mm)$ be a $\mms$ as in \eqref{eq:mms}, and let $p,q \in X$. 

If $(X,\sfd,\mm)$ is Lipschitz-infinitesimally strictly convex then, for $\mm$-a.e. $x \in X$, the angles $\bold{\angle_{pxq}}$ and $\bold{\angle_{qxp}}$ are single valued; i.e.  
$$\bold{\angle_{pxq}}= \{\angle_{pxq}\}, \quad \bold{\angle_{qxp}}= \{\angle_{qxp}\}.$$  

If moreover $(X,\sfd,\mm)$ is Lipschitz-infinitesimally Hilbertian then, for $\mm$-a.e. $x \in X$, not only the angles $\bold{\angle_{pxq}}$ and $\bold{\angle_{qxp}}$ are single valued but they also coincide; i.e.  
$$\bold{\angle_{pxq}}= \{\angle_{pxq}\}=\{\angle_{qxp}\}= \bold{\angle_{qxp}}.$$   
\end{theorem}

\subsection{The case $(X,\sfd,\mm)$ is a measured-Gromov-Hausdorff limit of Riemannian manifolds with Ricci curvature bounded from below}\label{Subsec:Honda}
In \cite{Honda12}, Honda gave a definition of angle on a Gromov-Hausdorff limit space of a sequence of complete $n$-dimensional Riemannian manifolds with a lower Ricci curvature bound and used that notion to introduce a weak second order differentiable structure on such spaces. The goal of the present subsection is to briefly recall the construction of Honda and to prove the equivalence almost everywhere with our definition. 

Throughout this subsection, $(X,\sfd, \mm, \bar{x})$ is a $\mms$ as in \eqref{eq:mms} which moreover is  a pointed measured-Gromov-Hausdorff limit space of a sequence of pointed complete $n$-dimensional Riemannian manifolds $\{(M_i,g_i,\bar{x}_i)\}_{i \in \N}$ with $\Ric_{(M_i,g_i)}\geq -(n-1)$. 
One of the main results in \cite{Honda11} and  \cite{Honda12}  is the following theorem. 
\begin{theorem}[Theorem 1.2 in \cite{Honda12} and Theorem 3.2 in \cite{Honda11}]\label{thm:Honda}
Let $p,q \in X\setminus\{\bar{x}\}$ with $\bar{x}\notin C_p\cup C_q$. Then one can define the angle $\angle(p,\bar{x},q)$ between $p$ and $q$ at $x$ as
\begin{equation}\label{eq:angle(pxq)}
\cos(\angle(p,\bar{x},q)):=\lim_{t\to 0} \frac{2t^2-\sfd^2(\gamma_{\bar{x},p}(t), \gamma_{\bar{x},q}(t))}{2t^2}, 
\end{equation}
for any unit speed  geodesic $\gamma_{\bar{x},p}$ from $\bar{x}$ to $p$ and $\gamma_{\bar{x},q}$ from $\bar{x}$ to $q$, where $C_p$ is the cut locus of $p$ defined by $C_p:=\{x \in X: \sfd(p,x)+\sfd(x,z)> \sfd(p,z) \text{ for every } z \in X\setminus\{x\}\}$.
\\Moreover,  the measure of the cut locus of any point of $X$ is null with respect to any limit measure $\mm$, therefore $\mm(C_p\cup C_q)=0$ and the angle $\angle(p,x,q)$ is well defined for $\mm$-a.e. $x \in X\setminus\{p,q\}$.
\end{theorem}	
Notice  that the definition of angle \eqref{eq:angle(pxq)} between $p$ and $q$ at $x$ is inspired by the corresponding notion in Alexandrov spaces \eqref{eq:angleAlex} discussed in the Introduction, but a priori \eqref{eq:angle(pxq)} is slightly weaker since Alexandrov definition asks that the joint limit for $s,t \to 0$ exists.

Honda also proved (see Corollary 4.2 in \cite{Honda12}) that for $p,q \in X$ and $\bar{x}\in X \setminus (C_p\cup C_q \cup\{p,q\})$ the limit
$$ \lim_{r\to 0} \frac{1}{\mm(B_r(\bar{x}))} \int_{B_r(\bar{x})} <\d r_p, \d r_q> \, \d \mm$$
exists and, by  using the Splitting Theorem, he shows that (see Theorem 4.3  in \cite{Honda12})
\begin{equation}\label{eq:anglem(pxq)}
\cos(\angle(p,\bar{x},q))=\lim_{r\to 0} \frac{1}{\mm(B_r(\bar{x}))} \int_{B_r(\bar{x})} <\d r_p, \d r_q> \, \d \mm,
\end{equation}  
$\angle(p,\bar{x},q)$ being defined in \eqref{eq:angle(pxq)}. Before proceeding with the discussions, let us briefly explain the meaning of the expression $<\d r_p, \d r_q>$ appearing in the last two formulas (for more details the interested reader may consult \cite{Cheeger00},  \cite{ChCo3} and \cite{HondaCAG}).

From the works of Cheeger \cite{Cheeger00} and Cheeger-Colding \cite{ChCo3} we know that $(X,\sfd,\mm, \bar{x})$, being a limit space as before, admits a cotangent bundle $T^* X$ endowed with a canonical scalar product (see Section 10 in \cite{Cheeger00} and Section 6 in \cite{ChCo3}; for a brief survey see also Section 2 in \cite{HondaCAG}). The cotangent bundle $T^* X$ satisfies the following properties:
\begin{enumerate}
\item $T^*X$ is a topological space.
\item There exists a Borel map $\pi: T^* X \to X $ such that $\mm(X\setminus \pi(T^* X)) = 0$.
\item $\pi ^{-1}  (x)$ is a finite-dimensional real vector space with a canonical scalar product $\nolinebreak[4]{<.,.>_x}$ for every $x\in \pi(T^* X)$.
\item For every open subset $U \subset X$ and every Lipschitz function $f$ on $U$, there exist a Borel subset $V \subset U$, and a Borel map $\d f$ (called the \emph{differential of $f$}) from $V$ to $T^* X$
such that $\mm(U \setminus V ) = 0$, 
\begin{equation}\label{eq:|df|=|Df|}
\pi \circ \d f (x) = x \quad \text{and} \quad |\d f |(x) = |Df|(x) \quad \text{for every } x \in V,
\end{equation} 
where $|v|(x) = \sqrt{<v,v>_x}$ and $|Df|$ is the slope defined in \eqref{eq:slope}.
\item For every open subset $U \subset X$ and every couple of Lipschitz functions $f,g$ on $U$, called  $V \subset U$ with $\mm(U \setminus V ) = 0$, $\d f, \d g$ the differentials of $f$ and $g$ from $V$ to $T^* X$ given by the previous point, the following holds: for every $\alpha, \beta \in \R$ the differential of the function 
$\alpha f+ \beta g$ is defined on $V$ and moreover 
\begin{equation}\label{eq:d(af+bg)}
\d(\alpha f+ \beta g)(x)=\alpha \d f (x)+ \beta \d g(x) \quad \text{for every } x \in V.
\end{equation}  
\end{enumerate}
Now we are in position to prove an easy but fundamental lemma about the equivalence $\mm$-a.e. of the infinitesimal Hilbertian structure of Cheeger-Colding  with the one of Definition \ref{def:Infin}. Recall that, being a limit space, $(X,\sfd,\mm)$ is Lipschitz-infinitesimally Hilbertian (see Remark \ref{rem:ExIHS}) so that for every $f,g \in \LIP(X)$ we have
\begin{equation}\label{eq:DfDg}
Df \cdot Dg:=D^+f (\nabla g)=D^-f (\nabla g)=D^+g (\nabla f)=D^-g (\nabla f) \quad \mm-a.e.,
\end{equation}
the object $D^\pm f(\nabla g)$ being introduced in Definition \ref{def:Dfnablag}.

\begin{lemma}[Equivalence of the infinitesimal Hilbertian structures]\label{lem:<>=Dnabla}
Let $(X,\sfd, \mm, \bar{x})$ be a $\mms$ as in \eqref{eq:mms} which moreover is  a pointed measured-Gromov-Hausdorff limit space of a sequence of pointed complete $n$-dimensional Riemannian manifolds $\{(M_i,g_i,\bar{x}_i)\}_{i \in \N}$ with $\Ric_{(M_i,g_i)}\geq -(n-1)$. Then for every $f,g:X \to \R$ locally Lipschitz it holds
\begin{equation}\label{eq:<>=Dnabla}
<\d f, \d g>_x= Df \cdot Dg (x) \quad \text{for $\mm$-a.e. } x \in X.
\end{equation}
\end{lemma}

\begin{proof}
For simplicity of notation let us prove the Lemma for $f,g \in \LIP(X)$, for the general case the arguments are analogous. By the property 4 of the cotangent bundle, given $f,g \in \LIP(X)$ we know that there exist $V\subset X$ with $\mm(X \setminus V ) = 0$ and $\d f, \d g$ differentials of $f$ and $g$ from $V$ to $T^* X$ satisfying \eqref{eq:|df|=|Df|}. Up to a further $\mm$-negligible subset, we can assume that on $V$ also the equalities \eqref{eq:DfDg} hold. So for every $x \in V$ we have
\begin{eqnarray}
Df \cdot Dg (x)&=&D^+f (\nabla g)(x)=\liminf_{\varepsilon \downarrow 0} \frac{|D(f+\varepsilon g)|^2(x)-|Df|^2(x)}{2 \varepsilon} \nonumber \\
&=&\liminf_{\varepsilon \downarrow 0} \frac{|\d(f+\varepsilon g)|^2(x)-|\d f|^2(x)}{2 \varepsilon}= \liminf_{\varepsilon \downarrow 0} \frac{|\d f+\varepsilon \d g|^2(x)-|\d f|^2(x)}{2 \varepsilon} \nonumber \\
&=& <\d f, \d g>_x\quad , \nonumber
\end{eqnarray} 
where the second equality is just Definition \ref{def:Dfnablag} with $p=2$, the third equality comes from property 5 of $T^*X$ together with \eqref{eq:|df|=|Df|} applied to the functions $f$ and $f+\varepsilon g$, the fourth equality is a consequence of \eqref{eq:d(af+bg)}, and the last equality follows by the Hilbertianity of the norm, i.e. $|\d f|^2(x)=<\d f, \d f>_x$, stated in property 4 of $T^*X$.
\end{proof}
In order to prove the equivalence of the definitions of the angles we need the following lemma (for the proof see Theorem 1.8 in \cite{Hein})
\begin{lemma}[Lebesgue's differentiation in doubling metric measure spaces]\label{lem:Lebesgue}
If $f$ is a  locally integrable function on a doubling $\mms$ $(X,\sfd,\mm)$, then 
$$
\lim_{r \to 0} \frac{1}{\mm(B_r(x))}\int_{B_r(x)} f \, \d \mm= f(x) \quad \text{for } \mm\text{-a.e. } x \in X.
$$
\end{lemma}
Now we are ready to prove the $\mm$-a.e. equivalence of the definitions of angles.
\begin{theorem}[$\mm$-a.e. equivalence of the definition of angles]\label{thm:equivalenceAngle}
Let $(X,\sfd, \mm, \bar{x})$ be a $\mms$ as in \eqref{eq:mms} which moreover is  a pointed measured-Gromov-Hausdorff limit space of a sequence of pointed complete $n$-dimensional Riemannian manifolds $\{(M_i,g_i,\bar{x}_i)\}_{i \in \N}$ with $\Ric_{(M_i,g_i)}\geq -(n-1)$. 

Then, fixed $p,q \in X$, for $\mm$-a.e. $x\in X\setminus\{p,q\}$ our notion of angle $\bold{\angle_{pxq}}$ given in Definition \ref{def:anglecone(X,d)}  coincides with the one of Honda given in \eqref{eq:angle(pxq)}:
\begin{equation}\label{thm:equivalenceAngle}
\bold{\angle_{pxq}}=\{\angle(p,x,q)\}.
\end{equation}
\end{theorem}

\begin{proof}
It is well known that limit spaces are doubling on bounded subsets (a possible way to see this,  maybe not the shortest,  is that  limit spaces are $CD(K,N)$ for some $K,N \in \R$ and then they are doubling by \cite{Sturm06II} Corollary 2.4), so the Lebesgue differentiation Lemma \ref{lem:Lebesgue} holds.

Fix $p,q \in X$. Combining Theorem \ref{thm:Honda} and formula \eqref{eq:anglem(pxq)}, we get that for $\mm$-a.e. $x \in X\setminus\{p,q\}$ the \textquotedblleft Honda\textquotedblright-angle $\angle(p,x,q)$ is given by the following expression
\begin{equation}\label{eq:proof1}
\cos(\angle(p,x,q))=\lim_{r\to 0} \frac{1}{\mm(B_r(x))} \int_{B_r(x)} <\d r_p, \d r_q> \, \d \mm.
\end{equation}   
By the equivalence of the infinitesimal Hilbertian structures stated in Lemma \ref{lem:<>=Dnabla}, we also know that $<\d r_p, \d r_q>=D r_p \cdot D r_q$ $\mm$-a.e.; hence we can rewrite \eqref{eq:proof1} as 
\begin{equation}\nonumber\label{eq:proof2}
\cos(\angle(p,x,q))=\lim_{r\to 0} \frac{1}{\mm(B_r(x))} \int_{B_r(x)} D r_p \cdot D r_q \, \d \mm.
\end{equation} 
From \eqref{Schwartz} the function $|D r_p \cdot D r_q|$ is bounded by 1 hence, being $\mm$  locally finite by assumption, locally integrable. By Lemma \ref{lem:Lebesgue} and recalling the definition of angle cone \eqref{eq:anglecone(X,d)} (see also \eqref{eq:cpm} \eqref{eq:angle} and \eqref{eq:DfDg}) we finally obtain
$$
\cos(\angle(p,x,q))=\lim_{r\to 0} \frac{1}{\mm(B_r(x))} \int_{B_r(x)} D r_p \cdot D r_q \, \d \mm=  D r_p \cdot D r_q(x)= \cos (\angle_{pxq}), 
$$
for $\mm$-a.e.  $x \in X\setminus\{p,q\}$ as desired. 
\end{proof}

\begin{remark}\label{rem:Honda+}
{\rm Fix $p,q \in X$. Let us remark that, even if by Theorem \ref{thm:equivalenceAngle} our notion of angle is equivalent to the Honda's one for $\mm$-a.e. $x \in X$, the latter does not work for  $x \in C_p\cup C_q$; on the contrary our  Definition \ref{def:anglecone(X,d)} permits to define an angle cone for every $x \in X$, so in particular also for $x \in C_p\cup C_q$.} \fr
\end{remark}
%Journals:  crelle, commentarii math helv.,  annali classe di scienze, analysis and geometry in metric spaces

%2)to compare our notion of angle and the one in Alexandrov  spaces we have to localize our notion.  Let $\gamma_{xp}, \gamma_{xq}$ be two length minimizing  geodesics connecting $x$ to $p$, resp. $x$ to $q$ and define
%$$\tilde{\angle}_{pxq}= \lim_{t \to 0} \nabla r_{\gamma_{xp}(t)} \nabla r_{\gamma_{xq}(t)} (x).$$

%3) we will show that in a regular point alla Burago it coincides with the alexandrov angle (just because it coincide in Euclidean space and look at local structure of Alexandrov spaces)

%\section{ The case of $CD(0,N)$-infinitesimally Hilbertian metric measure spaces }

%Usare blow up procedure  e spliting theorem di Nicola per provare  che fixed $\mm$-a.e. $p,x,q \in X$ we have 
%where $\gamma_{xp}$ (resp. $\gamma_{xq}$) are constant speed geodesics from $x$ to $p$ (resp. from $x$ to $q$) which can be extended to a left neighboorhod of $0$ still keeping the length minimizing property. For the existence of such geoedesics see Gigli GAFA. 

\def\cprime{$'$}

\end{document}